\documentclass[reqno]{amsart}

\usepackage[applemac]{inputenc}

\usepackage{amssymb}
\usepackage{graphicx}
\usepackage[cmtip,all]{xy}
\usepackage{verbatim}
\usepackage[toc,page]{appendix}
\usepackage{graphicx}
\usepackage{mathrsfs}
\usepackage{geometry}

\DeclareMathAlphabet{\mathpzc}{OT1}{pzc}{m}{it}

\newtheorem{theorem}{Theorem}[section]
\newtheorem{lemma}[theorem]{Lemma}

\theoremstyle{definition}
\newtheorem{definition}[theorem]{Definition}
\newtheorem{example}[theorem]{Example}

\theoremstyle{remark}

\numberwithin{equation}{section}

 \newcommand{\virgolette}{``}

\newcommand{\slantone}[2]{{\raisebox{.1em}{$#1$}\left/\raisebox{-.1em}{$#2$}\right.}}
\newcommand*{\defeq}{\mathrel{\vcenter{\baselineskip0.5ex \lineskiplimit0pt
                     \hbox{\scriptsize.}\hbox{\scriptsize.}}}%
                     =}

\newcommand{\slanttwo}[2]{{\raisebox{.2em}{$#1$} \big/ \raisebox{-.2em}{$#2$}}}

\newcommand\asim{\mathrel{%
  \ooalign{\raise0.1ex\hbox{$\sim$}\cr\hidewidth\raise-0.8ex\hbox{\scalebox{0.9}{$\scriptstyle{x}$}}\hidewidth\cr}}}

\newcommand{\proj}[1]{\ensuremath{\mathbb{P}^{#1}}}

\newcommand{\mani}{\ensuremath{\mathpzc{M}}}
\newcommand{\manir}{\ensuremath{\mathpzc{M}_{red}}}
\newcommand{\stsheaf}{\ensuremath{\mathcal{O}_{\mathpzc{M}}}}
\newcommand{\stsheafred}{\ensuremath{\mathcal{O}_{\mathpzc{M}_{red}}}}

\newcommand{\beq}{\begin{equation}}
\newcommand{\eeq}{\end{equation}}
\newcommand{\bear}{\begin{eqnarray}}
\newcommand{\eear}{\end{eqnarray}}

\begin{document}

\title{SUPERGEOMETRY OF $\Pi$-PROJECTIVE SPACES}


\author{SIMONE NOJA}
\address{Dipartimento di Matematica, Università degli Studi di Milano, Via Saldini 50, 20133 Milano, Italy}
\email{simone.noja@unimi.it}




\begin{abstract} 
In this paper we prove that $\Pi$-projective spaces $\proj n_\Pi$ arise naturally in supergeometry upon considering a non-projected thickening of $\proj n$ related to the cotangent sheaf $\Omega^1_{\proj n}$. In particular, we prove that for $n \geq 2$ the $\Pi$-projective space $\proj n_\Pi$ can be constructed as the non-projected supermanifold determined by three elements $(\proj n, \Omega^1_{\proj n}, \lambda)$, where $\proj n$ is the ordinary complex projective space, $\Omega^1_{\proj n}$ is its cotangent sheaf and $\lambda $ is a non-zero complex number, representative of the fundamental obstruction class $\omega \in H^1 (\mathcal{T}_{\proj n} \otimes \bigwedge^2 \Omega^1_{\proj n}) \cong \mathbb{C}.$ Likewise, in the case $n=1$ the $\Pi$-projective line $\proj 1_\Pi$ is the split supermanifold determined by the pair $(\proj 1, \Omega^1_{\proj 1} \cong \mathcal{O}_{\proj 1} (-2)).$ Moreover we show that in any dimension $\Pi$-projective spaces are Calabi-Yau supermanifolds. To conclude, we offer pieces of evidence that, more in general, also $\Pi$-Grassmannians can be constructed the same way using the cotangent sheaf of their underlying reduced Grassmannians, provided that also higher, possibly fermionic, obstruction classes are taken into account. This suggests that this unexpected connection with the cotangent sheaf is characteristic of $\Pi$-geometry.  

\end{abstract}

\maketitle

\tableofcontents

\section{Introduction: Supergeometry and $\Pi$-Projective Geometry}

The discovery of \emph{supersymmetry} in the Seventies brought a great deal of attention on \emph{supermathematics}. In this context, the \virgolette russian school'', after having spawned this research area by the pioneering works of F. Berezin and, later, of D. Leites, kept being a driving force during the Eighties. In particular, triggered by the rapid developments and the interest in the back then newborn \emph{superstring theory}, Yu.I. Manin, together with his former students I.B Penkov, I.A  Skornyakov and A.M. Levin, addressed the problem of lying down sound algebraic geometric foundations for supergeometry. This culminated in two dense books, \cite{Manin} and \cite{ManinNC}.

The work of Manin and his collaborators made clear that, if on the one hand it is true that many results in ordinary algebraic geometry have straightforward generalisations to a supergeometric context, on the other hand there are some striking exceptions. Among these, the most notable it is certainly the theory of differential forms on supermanifolds and the related integration theory. This is currently an active research area due to its close relation to superstring perturbation theory and certain supergravity formulations. A less-known remarkable difference with ordinary algebraic geometry is concerned with the role of \emph{projective superspaces}, the obvious generalisation of projective spaces, as a natural set-up and ambient space. Indeed, there are important examples of supermanifolds that fail to be \emph{projective}, i.e. they do not posses any invertible sheaf that allows for an embedding into projective superspaces \cite{PenkovSk}. For example, there is no natural generalisation of the Pl\"{u}cker map, so that super Grassmannians cannot in general be embedded into projective superspaces. This led Manin to suggest that in a supergeometric setting, invertible sheaves might not play the same fundamental role they play in ordinary algebraic geometry. Instead, together with Skornyakov, he proposed as a suitable substitute of invertible sheaves in algebraic supergeometry, the notion of $\Pi$\emph{-invertible sheaves}. These are locally-free sheaves of rank $1|1$ endowed with a specific \emph{odd} symmetry, locally exchanging the even and odd components, called $\Pi$\emph{-symmetry.} The spaces allowing for such sheaves to be defined were first constructed by Manin, see \cite{Manin}: these are called $\Pi$-projective spaces $\proj {n}_\Pi$ and more in general $\Pi$-Grassmannians. The relevance of these geometric objects became apparent along with the generalisation to a supersymmetric context of the theory of elliptic curves and theta functions due to Levin. In particular, it was realised in \cite{Levin1} and \cite{Levin2} that the correct supergeometric generalisation of theta functions, called \emph{supertheta functions}, should not be sections of a certain invertible sheaves, but instead sections of $\Pi$-invertible sheaves and every elliptic supersymmetric curves can be naturally embedded into a certain product of $\Pi$-projective spaces $\proj n_\Pi$ by means of supertheta functions. Recently, following an observation due to Deligne, Kwok has provided in \cite{Kwok} a different description of $\Pi$-projective spaces $\proj {n}_\Pi$ by constructing them as suitable quotients by the algebraic supergroup $\mathbb{G}^{1|1}_m = \mathbb{D}^\ast$, the multiplicative version of the \emph{super skew field} $\mathbb{D}$, which is a non-commutative associative superalgebra, thus making apparent a connection between $\Pi$-projective geometry and the broader universe of non-commutative geometry. 

In this paper we will provide a new construction of $\Pi$-projective spaces $\proj {n}_\Pi$, showing how they arise naturally as non-projected supermanifolds over $\proj {n}$, upon choosing the \emph{fermionic sheaf} of the supermanifold to be the cotangent sheaf $\Omega^1_{\proj n}$. More precisely, we will show that for $n > 1$ $\Pi$-projective spaces can be defined by three ordinary objects, a projective space $\proj n$, the sheaf of 1-forms $\Omega^1_{\proj n}$ defined on it and a certain cohomology class, called \emph{fundamental obstruction class}, $\omega \in H^1 (\proj {n}, \mathcal{T}_{\proj {n}} \otimes \bigwedge^2 \Omega^1_{\proj n}),$ where $\mathcal{T}_{\proj n}$ is the tangent sheaf of $\proj n$. In the case $n=1$ one does not need any cohomology class and the data coming from the projective line $\proj 1$ and the cotangent sheaf $\Omega^1_{\proj 1} = \mathcal{O}_{\proj 1}(-2)$ are enough to describe the $\Pi$-projective line $\proj 1_{\Pi}.$ Moreover we show that $\Pi$-projective spaces are \emph{all} Calabi-Yau supermanifolds, that is, they have trivial Berezinian sheaf, a feature that makes them particularly interesting for physical applications. 

Finally, we provide some pieces of evidence that the relation with the cotangent sheaf of the underlying manifold is actually a characterising one in $\Pi$-geometry. Indeed, not only $\Pi$-projective spaces, but also more in general $\Pi$-Grassmannians can be constructed as certain non-projected supermanifolds starting from the cotangent sheaf of the underlying reduced Grassmannias. We show by means of an example that in this context the non-projected structure of the supermanifold becomes in general more complicated and, in addition to the fundamental one, also higher obstruction classes enter the description.

The paper is organised as follows: in the first section we briefly review some basic materials about (complex) supermanifolds, we give some elements of $\Pi$-projective geometry and, following \cite{Manin}, we introduce the $\Pi$-projective spaces $\proj{n}_\Pi$ as closed sub-supermanifolds of certain super Grassmannians, whose construction will also be explained in the section. Having introduced the concept of non-projected supermanifold and the obstruction class in cohomology, detecting whether a supermanifold is non-projected or not, we finally construct the $\Pi$-projective spaces as the non-projected supermanifolds arising from the cotangent sheaf over $\proj n$ and a certain choice of a representative for the cohomology class obstructing the splitting of the supermanifold. We then show that $\Pi$-projective spaces have trivial Berezinian sheaf, thus proving they are examples of non-projected Calabi-Yau supermanifolds. In the last section we briefly hint at what happens in the more general context of $\Pi$-Grassmannians, by discussing the example of $G_{\Pi} (2,4)$. \\

\noindent {\bf Acknowledgments:} I am in debt with my thesis supervisor Bert van Geemen for pointing out the details of a construction explained in the Appendix of this paper. I thank Sergio Cacciatori and Riccardo Re for their help, support and suggestions. 

\section{Super Grassmannians and $\Pi$-Projective Geometry}

By referring to \cite{Manin} as general reference for the theory of supermanifolds, we shall just give the most important definitions in order to establish some terminology and notation. Note that we will always be working in the (super)analytic category and so we let our ground field be the complex numbers $\mathbb{C}.$
\begin{definition}[Complex Supermanifold] A complex supermanifold is a locally ringed space $(\mani, \stsheaf)$, where $\mani$ is a topological space and $\stsheaf$ is a sheaf of supercommutative rings on $\mani.$ We call $\stsheaf$ the \emph{structure sheaf} of the supermanifold.
\end{definition}
\noindent If we let $\mathcal{J}_\mani$ be the \emph{ideal of nilpotents} contained in $\stsheaf$, then the following conditions are satisfied: 
\begin{itemize} 
\item the pair $(\mani, \stsheafred)$, where $\stsheafred \defeq \stsheaf / \mathcal{J}_\mani$ defines a \emph{complex manifold}, called the \emph{reduced space} of the supermanifold (\mani, \stsheaf). This corresponds to the existence of a morphism of supermanifolds $\iota : \manir \rightarrow \mani$ for every supermanifold $\mani$, such that $\iota$ is actually a pair $\iota \defeq (\iota, \iota^\sharp)$, with $\iota: \mani \rightarrow \mani$ the identity on the underlying topological space and $\iota^\sharp : \stsheaf \rightarrow \stsheafred$ is the quotient map by the ideal of nilpotents;
\item the quotient $\mathcal{J}_\mani / \mathcal{J}_\mani^2$ defines a locally-free sheaf of $\stsheafred$-modules and it is called the \emph{fermionic sheaf}. Notice that the fermionic sheaf has rank $0|q$, if $q$ is the odd dimension of the supermanifold $\mani$; 
\item the structure sheaf $\stsheaf$ is \emph{locally} isomorphic to the exterior algebra $\bigwedge^\bullet_{\stsheafred} \mathcal{F}_\mani$, seen as a $\mathbb{Z}_2$-graded algebra. A supermanifold whose structure sheaf is given by an exterior algebra is said to be \emph{split}. If this is not the case, the supermanifold is called \emph{non-projected}. 
\end{itemize}
In what follows we will simply denote a supermanifold by $\mani$ and its reduced space with $\manir$. An important class of example of (split) complex supermanifolds is provided by the projective superspaces, $\proj {n|m} \defeq (\proj n, \bigwedge^\bullet \left (\mathbb{C}^m \otimes_{\mathbb{C}} \mathcal{O}_{\proj n}(-1) \right ))$. As in ordinary complex algebraic geometry, also in algebraic supergeometry projective superspaces $\proj {n|m}$ are particular cases of super Grassmannians. Indeed, the supersymmetric generalisation of an ordinary Grassmannians is a space parametrising $a|b$-dimensional linear subspaces of a given $n|m$-dimensional space $V^{n|m}$. In what follows, we will consider the case $V^{n|m}$ given by a vector superspace isomorphic to $\mathbb{C}^{n|m}$ and we briefly review the construction of Grassmannian supermanifolds by the patching technique, via their \virgolette big cells'' description. For concreteness, we will follow closely \cite{CNR}, inviting the reader willing to have a more general and detailed treatment of relative super Grassmannians to refer in particular to \cite{Manin} and to \cite{ManinNC} for super flag varieties.

We take $\mathbb{C}^{n|m}$ be such that $n|m = c_0|c_1 + d_0|d_1$. One can then view $\mathbb{C}^{n|m}$ as $\mathbb{C}^{c_0 + d_0} \oplus (\Pi \mathbb{C})^{c_1 + d_1}$, where $\Pi $ is the parity changing functor that indicates the reversing of the parity, and since $\mathbb{C}^{n|m}$ is freely-generated one can write its elements as row vectors with respect to a certain basis, \bear
 \mathbb{C}^{n|m} = \mbox{Span} \{ e_1^{0}, \ldots, e_n^{0} | e^{1}_1, \ldots, e^{1}_m \},\eear
 where the upper indices refer to the $\mathbb{Z}_2$-parity.
Now, we take a collection of indices, call it $I = I_0 \cup I_1$, such that $I_0 $ is a collection of $d_0 $ out of the $n $ indices of $\mathbb{C}^n$ and likewise $I_1$ is a collection of $d_1$ indices out of $m$ indices of $\Pi \mathbb{C}^m$. Then, if $\mathcal{I}$ is the set of such collections of indices $I$ one finds that $
\mbox{card}(\mathcal{I}) = \mbox{card} (\mathcal{I}_{0} \times \mathcal{I}_{1}) = {n \choose d_0 } \cdot {m \choose d_1 }:$
this will be the number \emph{super big cells} that cover the super Grassmannian.

We then associate a set of even and odd (complex) variables $\{ x^{\alpha \beta}_{I} \,| \, \xi^{\alpha \beta}_I \}$ to each element $I \in \mathcal{I}_I$: these even and odd variables can be arranged to fill in the places of a $ d_0|d_1 \times n|m = a|b \times (c_0 + d_0) | (c_1+d_1) $ super matrix in a way such that the columns having indices in $I \in \mathcal{I}_I$ forms a $(d_0 + d_1) \times (d_0 + d_1 )$ unit matrix, as follows 
\begin{equation}
\mathcal{Z}_{{I}} \defeq 
\left (
\begin{array}{ccc|ccc||ccc|ccc}
& &  & 1 \; & & & & & & & & \\
\;  & \; x_I \; & \; &  & \ddots & & & 0 & &\;  &\; \xi_I \; &\;  \\
& & & & & \; 1& & & & & & \\
\hline \hline
& & & & & & 1\; & & & & & \\
&\; \xi_I \; & & & 0& & &\ddots & & &\;  x_I \; & \\
& & & & & & & &\; 1 & & & 
\end{array}
\right ),
\end{equation}

We define the superspace $\mathcal{U}_I \rightarrow \mbox{Spec}\,  \mathbb{C} \cong \{ pt\}$ as the analytic superspace $\{ pt \} \times \mathbb{C}^{d_0 \cdot c_0 + d_1 \cdot c_1 | d_0\cdot c_1 + d_1\cdot c_0} \cong \mathbb{C}^{d_0 \cdot c_0 + d_1 \cdot c_1 | d_0\cdot c_1 + d_1\cdot c_0} $, having $\{ x^{\alpha \beta }_I \, | \, \xi^{\alpha \beta}_I \}$ as the complex coordinates over the point. The superspace $\mathcal{U}_I$ is called a \emph{super big cell} of the super Grassmannian when it is represented as explained above, via the super matrix $\mathcal{Z}_I$.

Two superspaces $\mathcal{U}_I$ and $\mathcal{U}_J$ for two different $I, J \in \mathcal{I}$ can be glued together as follows. Given $\mathcal{Z}_I$, the super big cell of $\mathcal{U}_I$, one considers the super submatrix $\mathcal{B}_{IJ}$ formed by the columns having indices in $J$ and we let $\mathcal{U}_{IJ} \defeq \mathcal{U}_I \cap \mathcal{U}_J$ be the maximal sub-superspace of $\mathcal{U}_I$ such that on $\mathcal{U}_{IJ}$ we have that $\mathcal{B}_{IJ}$ is invertible. Notice that for this to be true it is sufficient that the two determinants of the even parts of the matrix $\mathcal{B}_{IJ}$ are different from zero, as the odd parts do not affect invertibility. On the superspace $\mathcal{U}_{IJ}$ one has two sets of coordinates, $\{x^{\alpha \beta}_I \, | \, \xi^{\alpha \beta}_{I} \}$ and $\{ x^{\alpha \beta}_J \, |\, \xi^{\alpha \beta}_J \}$. One can pass from one system of coordinates to the other by the transformation rule $\mathcal{Z}_J  = \mathcal{B}^{-1}_{IJ} \mathcal{Z}_I.$ A concrete example of this procedure is provided in \cite{CNR}.

Glueing together all the superspaces $\mathcal{U}_I$ we obtain the super Grassmannian $G(d_0|d_1; {n|m})$. This is given as the quotient 
\bear
G (d_0|d_1;{n|m}) \defeq \slanttwo{\bigcup_{I \in \mathcal{I}}}{\mathcal{R}},
\eear 
where $\mathcal{R}$ are the equivalence relations generated by the change of coordinates described above. As described in \cite{Manin} (Chapter 4, \S{} 3), the maps $\psi_{\mathcal{U}_I} : \mathcal{U}_I \rightarrow G(d_0|d_1;{n|m})$ are isomorphisms onto open sub-superspace of $G(d_0|d_1;{n|m})$: so that the super big cells provide a \emph{local} description of the super Grassmannian, in the same way as a usual (complex) supermanifold is locally isomorphic to a superspace of the kind $\mathbb{C}^{n|m}.$  

As mentioned above, as in the ordinary complex algebraic geometric context, (complex) projective superspaces are the most immediate examples of super Grassmannians. Indeed, as one has $\proj {n} \defeq G(1; {n+1})$ in ordinary algebraic geometry, similarly one finds $\proj {n|m} \defeq G (1|0;  {n+1|m})$ in algebraic supergeometry.

We will now provide the reader with some basic notions of $\Pi$-projective geometry. As far as the author is concerned the most straightforward and immediate way to introduce $\Pi$-projective spaces and their geometry is via super Grassmannians. This approach has also the merit to make clear that $\Pi$-projective spaces are in general embedded in super Grassmannians, putting forward super Grassmannians as universal embedding spaces in a supergeometric context. \\
Our starting point is the definition of $\Pi$-symmetry. We will give a general definition on sheaves.
\begin{definition}[$\Pi$-Symmetry] Let $\mathcal{G}$ be a locally-free sheaf of $\stsheaf$-modules of rank $n|n$ on a supermanifold $\mani$, a $\Pi$-symmetry is an isomorphism such that $p_{\Pi} : \mathcal{G} \longrightarrow \Pi \mathcal{G}$ and such that $p_{\Pi }^2 = id.$
\end{definition}
We now work locally and use simply the supercommutative free module $\mathbb{C}^{n|n} = \mathbb{C}^{n} \oplus \Pi \mathbb{C}^{n}$, instead of a generic sheaf: we can therefore choose a certain basis of even elements such that $\mathbb{C}^n = \mbox{Span} \{ e_1, \ldots, e_n \} $ and we generate a basis for the whole $\mathbb{C}^{n|n}$ as follows
\begin{equation}
\mathbb{C}^{n|n} = \mbox{Span} \{ e_1, \ldots, e_n \,| \, p_{\Pi} e_1, \ldots, p_{\Pi}e_n \}.
\end{equation}
Clearly, the action of $p_\Pi$ exchanges the generators of $\mathbb{C}^n$ and $\Pi \mathbb{C}^n$.\\ 
We observe that somehow the presence of a $\Pi$-symmetry should remind us of a \virgolette physical supersymmetry'', as it transform even elements in odd elements and viceversa. Also, as supersymmetry requires a Hilbert space allowing for the the same amount of bosonic and fermionic states, similarly $\Pi$-symmetry imposes an equal number of even and odd dimensions for a certain \virgolette ambient space'', as it might be the supercommutative free module $\mathbb{C}^{n|n}$ above. \\
Along this line, one can give the following 
\begin{definition}[$\Pi$-Symmetric Submodule] Let $M$ be a supercommutative free $A$-module such that $M=A^n \oplus \Pi A^n $. Then we say that a super submodule $S\subset M$ is $\Pi$-symmetric if it is stable under the action of $p_\Pi.$   
\end{definition}
\noindent This has as a consequence the following obvious lemma: 
\begin{lemma}\label{submodules} Let $M$ be a supercommutative free $A$-module such that $M=A^n \oplus \Pi A^n$ together with a basis given by $\{ e_1, \ldots, e_n \,| \, p_{\Pi} e_1, \ldots, p_{\Pi}e_n \}$. Then a super submodule of M is $\Pi$-symmetric if and only if for every element $v = \sum_{i=1}^n x^i e_i + \xi^i p_{\Pi} e_i$ it also contains $v_{\Pi} = \sum_{i = 1}^n (- \xi^i e_i  + x^i p_\Pi e_i )$
\end{lemma}
\noindent The proof is clear, as $v_{\Pi}$ is nothing but the $\Pi$-transformed partner of $v$. Notice, though, the presence of a minus sign due to parity reasons.\\
Such $\Pi$-symmetric submodules allow us to define $\Pi$-projective superspaces, we call them $\proj{n}_{\Pi}$, and, more in general, $\Pi$-symmetric super Grassmannians. The construction follows closely the one of super Grassmannians of the kind $G(1|1; {n+1|n+1})$, but we only take into account $\Pi$-symmetric free submodules characterised as by the Lemma \ref{submodules}. These, in turn, allow us to write down the $n+1$ affine super cells covering the supermanifold $\proj {n}_\Pi$, each of these related to an affine supermanifold of the kind $\tilde{\mathcal{U}_i} \defeq (\mathcal{U}_i, \mathbb{C}[z_{1i}, \ldots, z_{ni}, \theta_{1i}, \ldots, \theta_{ni}]) \cong \mathbb{C}^{n|n}$ and where $\mathcal{U}_i $ are the usual open sets covering $\proj n$. We try to make these considerations clear by considering the case of the the $\Pi$-projective line, we call it $\proj {1}_\Pi$.
\begin{example}[$\Pi$-Projective Line $\proj {1}_\Pi$] This is the classifying space of the $\Pi$-symmetric $1|1$ dimensional super subspaces of $\mathbb{C}^{2|2}$, corresponding to the super Grassmannian $G_\Pi (1|1; {2|2})$, where subscript refers to the presence of further $\Pi$-symmetry with respect to the ordinary case treated previously. This is covered by two affine superspaces, each isomorphic to $\mathbb{C}^{1|1}$, having coordinates in the super big-cells notation given by
\begin{align}
\mathcal{Z}_{\mathcal{U}_0} \defeq 
\left ( \begin{array}{ccc||ccc} 
1 & & x_0 & 0 & &  \xi_0 \\
\hline \hline
0 & & - \xi_0  & 1 & & x_0
\end{array}
\right ) \qquad \qquad
\mathcal{Z}_{\mathcal{U}_1} \defeq
\left ( \begin{array}{ccc||ccc} 
x_1 & & 1 & \xi_1 & &  0 \\
\hline \hline
-\xi_1 & & 0  & x_1 & & 1 
\end{array}
\right ).
\end{align}
It is then not hard to find the transition functions in the intersections of the charts either by means of allowed rows and column operation or by the method explained above. By rows and columns operations, for example, one finds: 
\begin{align}
& \left ( \begin{array}{cc||cc} 
1 &  x_0 & 0 &   \xi_0 \\
\hline \hline
0 & - \xi_0  & 1  & x_0
\end{array}
\right )  \stackrel{R_0 / x_0, R_1 /x_0}{\longrightarrow} 
\left ( \begin{array}{cc||cc} 
1/x_0 & 1 & 0 &   \xi_0/x_0 \\
\hline \hline
0 &  - \xi_0/x_0  & 1/x_0 & 1
\end{array}
\right ) \nonumber \\
& \left ( \begin{array}{cc||cc} 
1/x_0 & 1 & 0  &  \xi_0/x_0 \\
\hline \hline
0 & - \xi_0/x_0  & 1/x_0 & 1
\end{array}
\right ) \stackrel{R_0 - \xi_0/x_0 R_1 }{\longrightarrow} \left ( \begin{array}{cc||cc} 
1/x_0 & 1 & - \xi_0/x_0^2 & 0\\
\hline \hline
0 & - \xi_0/x_0  & 1/x_0 & 1
\end{array}
\right ) \nonumber \\
& \left ( \begin{array}{cc||cc} 
1/x_0  & 1 & - \xi_0/x_0^2 & 0\\
\hline \hline
0 & - \xi_0/x_0  & 1/x_0 & 1
\end{array}
\right ) \stackrel{R_1 + \xi_0/x_0 R_0 }{\longrightarrow} 
\left ( \begin{array}{cc||cc} 
1/x_0 & 1 & - \xi_0/x_0^2 & 0\\
\hline \hline
\xi_0/x_0^2 & 0  & 1/x_0 & 1
\end{array}
\right ) \nonumber.
\end{align}
One can then read the transition functions in the intersection of the affine charts, characterising the structure sheaf $\mathcal{O}_{\proj {1}_\Pi}$ of the $\Pi$-projective line: 
\begin{equation} 
x_1 = \frac{1}{x_0}, \qquad \qquad \xi_1 = - \frac{\xi_1}{x_0^2}.
\end{equation}
This leads to the conclusion that the $\Pi$-projective line $\proj {1|1}_\Pi $ is the $1|1$-dimensional supermanifold that is completely characterised by the pair $(\proj {1}, \mathcal{O}_{\proj 1} (-2) ).$ We underline that $\mathcal{O}_{\proj 1} (-2) \cong \Omega^1_{\proj 1}$ over $\proj 1$: we will see in what follows that this is not by accident. 
\end{example}
\noindent Before we go on we remark two facts.
First, one can immediately observe that the $\Pi$-projective line is substantially different compared with the projective superline $\proj {1|1}$. Indeed $\proj {1|1}$ is (completely) characterised by the pair $(\proj 1, \mathcal{O}_{\proj 1} (-1))$. Also, without going into details, we note that while $\proj {1|1}$ can be structured as \emph{super Riemann surface}, the $\Pi$-projective line $\proj 1_\Pi$ cannot. Indeed, the fermionic bundle $\mathcal{F}_{\proj {1|1}_\Pi} = (\mathcal{O}_{\proj 1_\Pi} )_1$ is given by $\mathcal{O}_{\proj 1}(-2)$ and this does not define a \emph{theta characteristic} on $\proj 1.$ Indeed, there is just one such, and it is given by $\mathcal{O}_{\proj 1} (-1)$, therefore the only genus zero super Riemann surface is given by the ordinary $\proj {1|1} = (\proj {1}, \mathcal{O}_{\proj 1} (-1))$. This has a certain importance in the mathematical formulation of superstring perturbation theory. \\
Secondly, we recall that when the odd dimension is 1, the pair $(\manir, \mathcal{F}_\mani)$, consisting into the reduced space and the fermionic sheaf completely characterises the supermanifold $\mani$, as clearly $(\mathcal{O}_\mani)_0 \cong \stsheafred$ and $\mathcal{J}_\mani \cong (\mathcal{O}_\mani)_1 \cong \mathcal{F}_\mani$ and multiplication in $\stsheaf = \stsheafred \oplus \mathcal{F}_\mani$ is defined by the action of $\stsheafred$ on $\mathcal{F}_\mani$ in a unique way, as $\mathcal{F}_\mani$ is an ideal such that $\mathcal{F}_\mani^2=0$. This is observed, for example, in \cite{Manin} (Chapter 4, \S{} 2, Proposition 8). \\
In case the supermanifold has odd dimension greater than 1 it is no longer true in general that the supermanifold is completely determined by the pair $(\manir, \mathcal{F}_\mani)$ - if this is the case, then the supermanifold is split -. We will discuss these issues in the next section of the paper. \\
For $n>1$, $\Pi$-projective spaces $\proj n_\Pi $ are not indeed determined just by the pair $(\manir, \mathcal{F}_\mani)$. In the following theorem we use the same method as above to write down the generic form of the transition functions of $\proj n_\Pi$: we will see that a certain nilpotent correction appears in the even transition functions.
\begin{theorem} \label{pitrans} Let $\proj n_{\Pi} \defeq (\proj n, \mathcal{O}_{\proj n_{\Pi}})$ be the $n$-dimensional $\Pi$-projective space and let $\tilde{\mathcal{U}_i} = (\mathcal{U}_i, \mathbb{C}[z_{ji}, \theta_{ji}]) \cong \mathbb{C}^{n|n}$ for $i=0,\ldots,n$, $j\neq i$ be the affine supermanifolds covering $\proj n_{\Pi}$. In the intersections $\mathcal{U}_i \cap \mathcal{U}_j$ for $0 \leq i < j \leq n+1$ the transition functions characterising $\mathcal{O}_{\proj n_{\Pi}}$ have the following form:
\begin{align} \label{transi1}
& \ell \neq i : \qquad z_{\ell j} = \frac{z_{\ell i}}{z_{j i}} + \frac{\theta_{j i} \theta_{\ell i}}{z^2_{ji}}, \qquad \theta_{\ell j} = \frac{\theta_{\ell i}}{z_{ji}} - \frac{z_{\ell i}}{z^2_{ji}} \theta_{ji};   \\
& \label{transi2} \ell = i : \qquad z_{i j} = \frac{1}{z_{ji}} , \qquad \theta_{i j} = - \frac{\theta_{ji}}{z^2_{ji}}.
\end{align}
\end{theorem}
\begin{proof} $\proj n_\Pi$ is covered by $n+1$ affine charts, whose coordinates are given in the super big cell notation by 
\begin{align}
\mathcal{Z}_{\mathcal{U}_i} = \left ( \begin{array}{ccccc||ccccc}
z_{1i} & \cdots & 1 & \cdots & z_{ni} \; & \theta_{1i} & \cdots & 0 & \cdots & \theta_{ni} \\
\hline \hline
- \theta_{1i} & \cdots & 0 & \cdots & - \theta_{ni} \; & z_{1i} & \cdots & 1 & \cdots & z_{ni} 
\end{array}
\right ), 
\end{align}
where the 1's and 0's sit at the $i$-th positions. Considering the super big cell $\mathcal{Z}_{\mathcal{U}_j}$ for $j \neq i $ one can find the transition functions by bringing $\mathcal{Z}_{\mathcal{U}_i}$ in the form of $\mathcal{Z}_{\mathcal{U}_j}$ by means of allowed rows and column operations (and remembering that it is not possible to divide by a nilpotent element) as done above in the case of $\proj 1_\Pi$. It is easily checked that this yields the claimed result.
\end{proof}
\noindent Later on in the paper we will see that the same transition functions characterising $\proj n_\Pi$ arise naturally upon the choice of the cotangent sheaf as the fermionic sheaf for a supermanifold over $\proj n$.

\section{A Nod to Non-Projected Supermanifolds}

We have anticipated in the previous section that not all of the supermanifolds can be looked at simply as exterior algebras over an ordinary manifold, that is, not all of the supermanifolds are split supermanifolds. There exist non-projected supermanifolds. We now recast these ideas in a more precise form. In particular, we have seen in the introduction that the structure sheaf $\stsheaf$ of a supermanifold can be seen as an \emph{extension} of $\stsheafred$ by $\mathcal{J}_\mani$, that is we have a short exact sequence
\bear
\xymatrix@R=1.5pt{ 
0 \ar[r] & \mathcal{J}_\mani \ar[r] &  \stsheaf \ar[r] & \stsheafred \ar[r] & 0.
 }
\eear
A very natural question that arises when looking at this exact sequence is whether is it \emph{split} or not. In other words, one might wonder whether there exists a morphism of supermanifolds, we call it $\pi : \mani \rightarrow \manir $, splitting the sequence as follows
\bear
\xymatrix@R=1.5pt{ 
0 \ar[r] & \mathcal{J}_\mani \ar[r] &  \stsheaf  \ar[r]_{\iota^\sharp} & \ar@{-->}@/_1.3pc/[l]_{\pi^\sharp} \stsheafred \ar[r] & 0.
 }
\eear
Supermanifolds that do posses a splitting morphism $\pi : \mani \rightarrow \manir$ such that $\pi \circ \iota = id_{\manir}$ are called \emph{projected}. \\
If there exists a projection $\pi : \mani \rightarrow \manir $ for $\mani$, then its structure sheaf is such that $\stsheaf \cong \stsheafred \oplus \mathcal{J}_\mani$ and it is a sheaf of $\stsheafred$-modules. This translates as follows at the level of the even transition functions of the (projected) supermanifold: considering a supermanifold of dimension $p|q$ with an atlas given by $\{ \mathcal{U}_i, z_{\kappa i } | \theta_{\ell i} \}_{ i \in I}$ where the $z_{\kappa i} | \theta_{\ell i}$ are the even and odd local coordinates, for $\kappa = 1, \ldots, p$ and $\ell = 1, \ldots, q$, in an intersection $\mathcal{U}_i \cap \mathcal{U}_j$, for the \emph{even} transition functions one simply gets
\bear
z_{\kappa i } (\underline z_{j}, \underline \theta_{j}) = z_{\kappa i } (\underline z_{j}),
\eear
where the $z_{\kappa i}$ are holomorphic as functions of $\underline {z}_j$. That is, the even transition functions of a projected supermanifold are those of an ordinary complex manifold and indeed they correspond to those of $\manir$. In particular, obviously, split supermanifolds are projected supermanifolds (while the converse is not in general true) and their even transition functions are of the kind above. 

In the case there is no such projection $\pi : \mani \rightarrow \manir$ we say that the supermanifold is \emph{non-projected} and its structure sheaf is not, in general, a sheaf of $\stsheafred$-modules. Keeping the same notation as above, in the intersection $\mathcal{U}_i \cap \mathcal{U}_j$ the \emph{even} transition functions of a non-projected supermanifold then get corrections by some (even) combination of nilpotent odd coordinates and they can be written as  
\bear \label{transf}
z_{\kappa i } (\underline z_{j}, \underline \theta_{j}) = z_{\kappa i } (\underline z_{j}) + \sum_{n=1}^{\lfloor q/2 \rfloor} \omega^{(2n)}_{ij} (\underline z_j, \underline \theta_j) \cdot z_{\kappa i}.
\eear
In the expression above $\omega_{ij}^{(2n)} \in Z^1 ( \mathcal{T}_{\manir} \otimes Sym^{2n} \mathcal{F}_\mani) (\mathcal{U}_i \cap \mathcal{U}_j)$ is a representative of \emph{$n$-th even obstruction class} $\omega^{(2n)} \in H^1 (\mathcal{T}_{\manir} \otimes Sym^{2n} \mathcal{F}_\mani).$ \\
A sufficient condition for a supermanifold to be non-projected is given by the following
\begin{theorem} Let $\mani$ be a complex supermanifold. If $ 0 \neq \omega_\mani \in H^1(\mani, \mathcal{T}_{\manir} \otimes Sym^2 \mathcal{F}_\mani) $ $\mani$, then $\mani$ is non-projected.
\end{theorem}
\begin{proof} this is the main result of \cite{Green}. For a more recent account see \cite{DonWit}. The interested reader might find a detailed proof in case the odd dimension of the supermanifold is equal to 2 in \cite{CNR}.
\end{proof}
\noindent Relying on this theorem, we call $\omega_\mani \defeq \omega^{(2)} \in H^1 (\mathcal{T}_{\manir} \otimes Sym^2 \mathcal{F}_\mani)$ the \emph{fundamental obstruction class} of the supermanifold $\mani$. \\
At this point, it is important to notice that the transition functions of the $\Pi$-projective space $\proj n_\Pi$ we have found in Theorem \ref{pitrans} are a typical example of transition functions of a non-projected supermanifolds: in particular, it is crucial to note the $\theta\theta$-correction in the even transition functions, 
\bear
z_{\ell j} = \frac{z_{\ell i}}{z_{j i}} + \frac{\theta_{j i} \theta_{\ell i}}{z^2_{ji}}.
\eear
The bit having form ${\theta \theta}/{z^2}$ is indeed the signal of certain cocycle representing the fundamental obstruction class $\omega_{\proj n_\Pi}.$ The identification of $\omega_{\proj n_\Pi}$ will play a crucial role as we shall see shortly.\\
Incidentally, before we go on, we stress that also the \emph{odd} transition functions might not be just linear in the nilpotent odd coordinates - that is, they might fail to be the transition functions of a locally-free sheaf of $\stsheafred$-modules. Again, this corresponds to the presence of certain non-null cohomology classes, but we will not dwell further into this topic as these \emph{odd obstruction classes} will not be used in the paper: we will only need the fundamental obstruction class $\omega_\mani \in H^1 (\mathcal{T}_{\manir} \otimes Sym^2 \mathcal{F}_\mani)$.

\section{Cotangent Sheaf, Obstructions and $\Pi$-Projective Spaces} 
\noindent For future use we start fixing the notation we are to adopt when working on ordinary projective spaces $\proj n$. We consider the usual covering by $n+1$ open sets $\{ \mathcal{U}_i \}_{i =0 }^n$ characterised by the condition $\mathcal{U}_i \defeq \{ [X_0 : \ldots : X_n] \in \proj n: X_i \neq 0\}$. Defining the affine coordinates to be
\bear
z_{ji} \defeq \frac{X_j}{X_i},
\eear
we have that $\proj n$ gets covered by $n+1$ affine charts as follows $(\mathcal{U}_i, \mathbb{C}[z_{1i}, \ldots, z_{ni}])$. This allows to easily write down the transition functions for two sheaves of interest, the tangent and the cotangent sheaf. 
\begin{itemize}
\item{\bf Tangent Sheaf} $\mathcal{T}_{\proj n}:$ on the intersection $\mathcal{U}_i \cap \mathcal{U}_j$ one finds:
\begin{align}
& \partial_{z_{ji}} = - z_{ij} \sum_{k \neq j} z_{kj} \partial_{z_{kj}} \\
& \partial_{z_{ki}} = z_{ij} \partial_{z_{kj}} \qquad k \neq j
\end{align}
\item { \bf Cotangent Sheaf} $\Omega^1_{\proj n}:$ on the intersection $\mathcal{U}_i \cap \mathcal{U}_j$ one finds:
\begin{align}
& dz_{ji} = - \frac{dz_{ij}}{z_{ij}^2}  \\
& dz_{ki} = - \frac{z_{kj}}{z_{ij}^2} dz_{ij} + \frac{dz_{kj}}{z_{ij}} \qquad k \neq j
\end{align}
\end{itemize}
We now consider a supermanifold of dimension $n|n$ having reduced space given by $\proj n$ and a fermionic sheaf $\mathcal{F}_\mani$ given by $\Pi \Omega^1_{\proj n}$. This is a sheaf of $\mathcal{O}_{\proj n}$-modules of rank $0|n$, that is locally-generated on $\mathcal{U}_i$ by $n$ odd elements $\{\theta_{1i}, \ldots, \theta_{ni} \}$ that transform on the intersections $\mathcal{U}_i \cap \mathcal{U}_j$ as the local generators of the cotangent sheaf, $\{dz_{1i}, \ldots, dz_{ni}\}$: in other words, the correspondence is $dz_{ki} \leftrightarrow \theta_{ki}$ for $k \neq i$. \\
The crucial observation is that \emph{the fermionic sheaf} $\mathcal{F}_\mani = \Pi \Omega^1_{\proj n}$ \emph{determined by the cotangent sheaf} $\Omega^1_{\proj n}$ \emph{reproduces exactly the odd transition functions of} $\proj n_{\Pi}.$ It is then natural to ask whether the whole $\mathcal{O}_{\proj n_{\Pi}}$ is determined someway by $\Pi \Omega^1_{\proj n}.$ We will see that this question has an affermative answer, by realising $\proj n_{\Pi}$ as the non-projected supermanifold whose even part of the structure sheaf $\mathcal{O}_{\proj n_{\Pi}}$ is determined by $\mathcal{O}_{\proj n}$ and the fundamental obstruction class $\omega_\mani.$\\
First we need to prove that the choice $\mathcal{F}_\mani = \Pi \Omega^1_{\proj n}$ can actually give rise to a non-projected supermanifold. For this to be true it is enough that $H^1(\mathcal{T}_{\proj n} \otimes Sym^2 \Pi \Omega^1_{\proj n} ) \neq 0$: this is achieved in the following
\begin{lemma} $H^1 (\mathcal{T}_{\proj n} \otimes Sym^2 \Pi \Omega^1_{\proj n}) \cong \mathbb{C}$
\end{lemma}
\begin{proof} It is important to realise that due to parity reason, one has $Sym^2 \Pi \Omega^1_{\proj n} \cong \bigwedge^2 \Omega^1_{\proj n}$, therefore it amounts to evaluate $H^1(\mathcal{T}_{\proj n} \otimes \bigwedge^2 \Omega^1_{\proj n}):$ this can be done using the Euler exact sequence tensored by $\Omega^2_{\proj n} \defeq \bigwedge^2 \Omega^1_{\proj n}, $ this reads
\bear
\xymatrix@R=1.5pt{ 
0 \ar[r] & \Omega^2_{\proj n}  \ar[r] & \Omega^2_{\proj n}(+1)^{\oplus n+1} \ar[r] & \mathcal{T}_{\proj n} \otimes \Omega^2_{\proj n} \ar[r] & 0.
 } 
\eear
Using Bott formulas (see for example \cite{OkScSp}) to evaluate the cohomology of $\Omega^2_{\proj n}$ and $ \Omega^2_{\proj n}(+1)$ one is left with the isomorphism $H^1 ( \mathcal{T}_{\proj n} \otimes \Omega^2_{\proj n}) \cong H^2 (\Omega^2_{\proj n}) \cong \mathbb{C}$, again by Bott formulas. 
\end{proof}
A consequence of the lemma is that each choice of a class $0 \neq \omega_\mani \in H^1 (\mathcal{T}_{\proj n} \otimes Sym^2 \Pi \Omega^1_{\proj n})$ gives rise to a non-projected supermanifold having reduced space $\proj n$ and fermionic sheaf $\Pi \Omega^1_{\proj n}.$ \\
Making use of the Bott formulas, as in proof of previous Lemma, is certainly the briefest and easiest to show the non vanishing of the cohomology group. Anyway, there is another more instructive way to achieve the same result: this, has also the merit to allow the right setting to find the representative of $H^1 (\mathcal{T}_{\proj n} \otimes Sym^2 \Pi \Omega^1_{\proj n})$. Keeping in mind that $Sym^2 \Pi \Omega^2_{\proj n} \cong \bigwedge^2 \Omega^1_{\proj n }$, one starts from the dual of the the Euler exact sequence, that reads
\bear
\xymatrix@R=1.5pt{ 
0 \ar[r] & \Omega^1_{\proj n}  \ar[r] & \mathcal{O}_{\proj n}(-1)^{\oplus n+1} \ar[r] &\mathcal{O}_{\proj n} \ar[r] & 0.
 } 
\eear
Taking its second exterior power one gets 
\bear \label{wedge2}
\xymatrix@R=1.5pt{ 
0 \ar[r] & \bigwedge^2 \Omega^1_{\proj n}  \ar[r] & \bigwedge^2 \left (\mathcal{O}_{\proj n}(-1)^{\oplus n+1} \right ) \ar[r] &\Omega^1_{\proj n} \ar[r] & 0.
 } 
\eear
Notice that clearly $\bigwedge^2 \mathcal{O}_{\proj n}(-1)^{\oplus n+1} \cong \mathcal{O}_{\proj n}(-2)^{\oplus {n+1 \choose 2}}$, and, more important, that the existence of this short exact sequence depends on the fact that $\mathcal{O}_{\proj n}$ is of rank $1$. A more careful discussion of the general framework for second exterior powers of short exact sequences of locally-free sheaf of $\mathcal{O}_{\manir}$-modules is deferred to the Appendix.\\
This short exact sequence can be in turn tensored by $\mathcal{T}_{\proj n}$ as to yield 
\bear \label{wedge2T}
\xymatrix@R=1.5pt{ \qquad 
0 \ar[r] & \mathcal{T}_{\proj n } \otimes \bigwedge^2 \Omega^1_{\proj n}  \ar[r] & \mathcal{T}_{\proj n } (-2)^{\oplus {n+1 \choose 2}} \ar[r] & \mathcal{T}_{\proj n } \otimes \Omega^1_{\proj n} \ar[r] & 0.
 } 
\eear
Upon using the Euler exact sequence for the tangent sheaf twisted by $\mathcal{O}_{\proj n} (-2)$, one sees that for $n >1$ the cohomology groups $H^0(\mathcal{T}_{\proj n } (-2)^{\oplus {n+1 \choose 2}})$ and $H^1(\mathcal{T}_{\proj n } (-2)^{\oplus {n+1 \choose 2}})$ are zero, and therefore the long exact cohomology sequence gives the isomorphism $H^0 (\mathcal{T}_{\proj n } \otimes \Omega^1_{\proj n}) \cong H^1 (\mathcal{T}_{\proj n} \otimes \bigwedge^2 \Omega^1_{\proj n}).$\\
The global section generating $H^0 (\mathcal{T}_{\proj n} \otimes \Omega^1_{\proj n})$ is easily identified as the diagonal element in $C^0 (\mathcal{T}_{\proj n} \otimes \Omega^1_{\proj n})$. By representing it locally, in the chart $\mathcal{U}_i$, one has
\bear
H^0 (\mathcal{T}_{\proj n} \otimes \Omega^1_{\proj n}) \cong \big \langle \sum_{j \neq i} \partial_{z_{ji}} \otimes dz_{ji} \big \rangle_{\mathbb{C}}.
\eear
and it is easily proved that it actually defines a global section.\\
We aim to lift this element to the generator of $H^1(\mathcal{T}_{\proj n} \otimes \bigwedge^2 \Omega^1_{\proj n})$, making the isomorphism explicit: this will be the key step of our construction of $\Pi$-projective spaces as non-projected supermanifolds. To achieve this, we need to study carefully the homomorphisms of sheaves entering the exact sequence \ref{wedge2}. First we consider the injective map $\wedge^2 \iota : \bigwedge \Omega^1_{\proj n} \rightarrow \bigwedge^2 \mathcal{O}_{\proj n} (-1)$. This is given by 
\bear
\xymatrix@R=1.5pt{ \qquad 
\wedge^2 \iota : \bigwedge^2 \Omega^1_{\proj n} \ar[rr] && \bigwedge^2 \left (\mathcal{O}_{\proj n} (-1)^{\oplus n+1}\right )\\
 df \wedge dg \ar@{|->}[rr] && \iota (df) \wedge \iota (dg)
 } 
\eear
where $\iota : \Omega^1_{\proj n} \rightarrow \mathcal{O} (-1)^{\oplus n+1}$ is the map that enters the dual of the Euler exact sequence, that is, working for example in the chart $\mathcal{U}_i$, 
\bear
\xymatrix@R=1.5pt{ \qquad 
\iota : \Omega^1_{\proj n} \ar[rr] && \mathcal{O}_{\proj n} (-1)^{\oplus n+1}\\
 df = \sum_{j \neq i} f_{ji} dz_{ji}  \ar@{|->}[rr]  && \left (\frac{f_{0i}}{X_i}, \ldots , - \frac{1}{X_i^2} \sum_{j \neq i} X_j f_{ji} , \ldots, \frac{f_{nj}}{X_i}\right ).
 } 
\eear
Getting back to \ref{wedge2} and keep working in the chart $\mathcal{U}_i$, the map $\Phi_2 :\bigwedge^2 \left ( \mathcal{O}_{\proj n}(-1)^{\oplus n+1} \right ) \rightarrow \Omega^1_{\proj n}$ is defined as follows
\bear
\xymatrix@R=1.5pt{ \nonumber
\Phi_2 : \bigwedge^2 \left ( \mathcal{O}_{\proj n}(-1)^{\oplus n+1} \right ) \ar[rr] &&\Omega^1_{\proj n}\\
 (f_0, \ldots, f_n) \wedge (g_0, \wedge, g_n) \ar@{|->}[rr]  && X_i \sum_{j=0}^n \sum_{k \neq i} X_j \left (f_j \otimes g_k -  g_j \otimes f_k \right ) d \left (\frac{X_k}{X_i}\right),
 } 
\eear
where $d (X_k / X_i ) = dz_{ki}.$ The reader can check that these maps give rise to an exact sequence of locally-free $\mathcal{O}_{\proj n}$-modules. Clearly, the maps entering the exact sequence \ref{wedge2T} are just the same tensored by identity on the tangent sheaf.\\
Upon knowing these map, we can prove the following lemma.
\begin{lemma}[Lifting] \label{lifting} The cohomology group $H^1 (\mathcal{T}_{\proj n} \otimes \bigwedge^2 \Omega^1_{\proj n})$ is represented by $\{ \omega_{ij} \}_{i < j } \in \prod_{i<j} \left ( \mathcal{T}_{\proj n} \otimes \bigwedge^2 \Omega^1_{\proj n} \right )(\mathcal{U}_i \cap \mathcal{U}_j)$, such that 
\bear
\omega_{ij} = \sum_{k \neq j} \frac{dz_{ij} \wedge dz_{kj}}{z_{ij}} \otimes \partial_{z_{kj}}.  
\eear
\end{lemma}
\begin{proof} We need to lift the element $(\sum_{j\neq i} \partial_{z_{ji}} \otimes dz_{ji})_{i = 0, \ldots, n} \in Z^0 (\mathcal{T}_{\proj n} \otimes \Omega^1_{\proj n})$ to $Z^{1}(\mathcal{T}_{\proj n} \otimes \bigwedge^2 \Omega^1_{\proj n})$ as in the following diagram: 
\bear
\xymatrix{  \nonumber 
C^{1} (\mathcal{T}_{\proj n} \otimes \bigwedge^2 \Omega^1_{\proj n}) \ar@{>->}[r] & C^1 (\mathcal{T}_{\proj n} \otimes \bigwedge^2 \mathcal{O}_{\proj n} (-1)^{\oplus n+1} ) & \\
& C^0 (\mathcal{T}_{\proj n} \otimes \bigwedge^2 \mathcal{O}_{\proj n} (-1)^{\oplus n+1} ) \ar@{>>}[r] \ar[u]^{\delta} & C^0 (\mathcal{T}_{\proj n} \otimes \Omega^1_{\proj n})
 } 
\eear 
where the maps are induced by those defining the short exact sequence \ref{wedge2T}. The first step is to find the pre-image of the element $\sum_{j\neq i} \partial_{z_{ji}} \otimes dz_{ji}$ in $C^0 (\mathcal{T}_{\proj n} \otimes \bigwedge^2 \mathcal{O}_{\proj n} (-1)^{\oplus n+1})$. We work, for simplicity, in the chart $\mathcal{U}_i$, and we look for elements $f$'s and $g$'s such that  
\bear \nonumber
\sum_{j \neq i} \partial_{z_{ji}} \otimes dz_{ji} \stackrel{!}{=} \left ( \sum_{j \neq i} \partial_{z_{ji}} \right ) \otimes \left ( \sum_{\ell \neq i} X_i \sum_{k = 0}^n X_k \left[ f_k^{(j)}  \otimes g^{(j)}_\ell - g_{k}^{(j)} \otimes f^{(j)}_{\ell} \right ] dz_{\ell i} \right ).
\eear 
The condition is satisfied by the choice
\bear
f^{(j)}_k = \frac{\delta_{ki}}{X_i} \qquad \qquad g^{(j)}_\ell = \frac{\delta_{\ell j}}{X_i},
\eear
so that one finds that the pre-image in the chart $\mathcal{U}_i$ reads
\bear
\Phi_2^{-1} \left (\sum_{j\neq i} \partial_{z_{ji}} \otimes dz_{ji} \right ) = \sum_{k \neq i} \partial_{z_{ki}} \otimes \frac{e_i \wedge e_k }{X_i^2}
\eear
where we have denoted $\{ e_i \wedge e_k \}_{i \neq k} $ a basis for the second exterior power $\bigwedge^{2} \mathcal{O}_{\proj n} (-1)^{\oplus n+1}.$ \\
Now we lift this element to a \v{C}ech 1-cochain by means of the \v{C}ech coboundary map $\delta$ as to get on an intersection $\mathcal{U}_i \cap \mathcal{U}_j$
\begin{align}
s_{ij} \defeq s_i - s_j \lfloor_{\mathcal{U}_i \cap \mathcal{U}_j} & =  \frac{1}{X_j^2} \sum_{k \neq i,j} \partial_{z_{kj}} \otimes  \left ( \frac{z_{kj}}{z_{ij}}  e_j \wedge e_i + \frac{1}{z_{ij}} e_i \wedge e_k + e_j \wedge e_k \right )  
\end{align}
It is not hard to verity that $(s_{ij})_{i \neq j}$ is the image through the injective map $\wedge^2 \iota$ of elements 
\bear
\omega_{ij}\defeq \sum_{k \neq j} \frac{dz_{ij} \wedge dz_{kj}}{z_{ij}} \otimes \partial_{z_{kj}} \in Z^1 (\mathcal{U}, \mathcal{T}_{\proj n} \otimes \bigwedge^2 \Omega^1_{\proj n}),
\eear
which represents the lifting of $\sum_{j \neq i} \partial_{z_{ji}} \otimes dz_{ji}$, and generates the cohomology group $H^1 (\mathcal{T}_{\proj n} \otimes \bigwedge^2 \Omega^1_{\proj n}) $. 
\end{proof}
Notice that one finds the actual elements that enter the transition functions of the non-projected supermanifold in the identification $dz_{ij} \wedge dz_{kj} \leftrightarrow \theta_{ij} \theta_{kj}$, where the second is to be understood as the symmetric product, giving an element in $Sym^2 \Pi \Omega^1_{\proj n}.$\\
The previous lemma gives all the elements we need in order to recognise the $\Pi$-projective space $\proj{n}_{\Pi}$ as the non-projected supermanifold is associated to the cotangent sheaf on $\proj n$. In particular, we have the following 
\begin{theorem}[$\Pi$-projective spaces] The $\Pi$-projected space $\proj n_\Pi \defeq (\proj n, \mathcal{O}_{\proj n_{\Pi}})$ is the non-projected supermanifold uniquely identified by the triple $(\proj n, \Pi \Omega^1_{\proj n}, \lambda)$, where $\lambda \neq 0$ is the representative of $\omega_\mani \in H^1(\mathcal{T}_{\proj n} \otimes Sym^2 \Pi \Omega^1_{\proj n}) \cong \mathbb{C}$. 
\end{theorem}
\begin{proof} It is enough to proof that the sheaf $\mathcal{O}_{\proj n_{\Pi}}$ can be determined out of the structure sheaf $\mathcal{O}_{\proj n}$ of $\proj n$, the fermionic sheaf $\Pi \Omega^1_{\proj n}$ and the non-zero representative $\lambda \in \mathbb{C} \setminus \{0\}$ of $\omega_\mani \in H^1 (\mathcal{T}_{\proj n} \otimes Sym^2 \Pi \Omega^1) \cong \mathbb{C}.$ We have already observed that the transition functions of $(\mathcal{O}_{\proj n_{\Pi}})_1$ do coincide with those of $\Pi \Omega^1_{\proj n}.$ Moreover, up to a change of coordinate or a scaling, $\lambda$ can be chosen equal to $1$. Then, we see that the transition functions of $(\mathcal{O}_{\proj n_\Pi})_0$ are determined by \ref{transf} as a non-projected extension of $\mathcal{O}_{\proj n}$ by $Sym^2 \Pi \Omega^1_{\proj n}$, as follows
\begin{align}
z_{ki} &= \frac{z_{kj}}{z_{ij}} + \left (\sum_{k\neq j}\frac{\theta_{ij}\theta_{kj}}{z_{ij}} \partial_{z_{kj}} \right ) z_{ki}  = \frac{z_{kj}}{z_{ij}} + \left (\sum_{k\neq j}\frac{\theta_{ij}\theta_{kj}}{z_{ij}^2} \partial_{z_{ki}} \right ) z_{ki} \\
& = \frac{z_{kj}}{z_{ij}} + \frac{\theta_{ij}\theta_{kj}}{z_{ij}^2}  
\end{align}
and clearly, $z_{ji} = 1/z_{ij}$. Here, we have used the result of the previous lemma \ref{lifting} to write the representatives of the fundamental obstruction class $\omega^{(2)}_{ij}$. This completes the proof. 
\end{proof}
Now that we have constructed $\Pi$-projective spaces as non-projected supermanifolds, we investigate a property that all of the $\Pi$-projective spaces share, regardless their dimensions: they have {trivial Berezinian sheaf}. Supermanifolds having this property are said Calabi-Yau supermanifolds, as the Berezianian sheaf is - in some sense - the only meaningful supersymmetric generalisation of the canonical sheaf. \vspace{5pt} 

\noindent Before we go into the proof of the theorem, we recall that given a supermanifold $\mani$, the Berezinian sheaf $\mathcal{B}er_{\mani}$ is, by definition the sheaf $\mathcal{B}er (\Omega^1_{\mani}) $. That is, given an open covering $\{ \mathcal{U}_i \}_{i \in \mathcal{I}}$ of the underlying topological space of $\mani$, the Berezinian sheaf $\mathcal{B}er_\mani$ is the sheaf whose transition functions $\{ g_{ij} \}_{i \neq j \in \mathcal{I}}$ are obtained by taking the Berezinian of the super Jacobian of a change of coordinates in $\mathcal{U}_{i} \cap \mathcal{U}_j$. 
\begin{theorem}[$\proj n_\Pi$ are Calabi-Yau supermanifolds] \label{SCY} $\Pi$-projective spaces $\proj n_\Pi$ have trivial Berezianian sheaf. That is, $\mathcal{B}er_{\proj n_\Pi} \cong \mathcal{O}_{\proj n_\Pi}.$
\end{theorem}
\begin{proof} Let us consider the generic case $n\geq 2$. It is enough to prove the triviality of the Berezianian sheaf in a single intersection. One starts computing the super Jacobian of the transition functions in \ref{transi1} and \ref{transi2}: this actually gives the transition functions of the cotangent sheaf to $\proj n_\Pi$ in a certain intersection $\mathcal{U}_{i} \cap \mathcal{U}_j$, that can be represented in a super matrix of the form 
\bear
[\mathcal{J}ac]_{ij} = \left ( \begin{array}{ccc|ccc}
 & & & & &  \\
 & A &  &  & B & \\
 & & & & &  \\
 \hline
 & & & & &  \\
 & C &  &  & D & \\
 &  &  &  &  & 
\end{array}
\right )
\eear
for some $A$ and $D$ even and $B$ and $C$ odd sub-matrices depending on the intersection $\mathcal{U}_i \cap \mathcal{U}_j$. Then one computes $\mbox{Ber}\, [\mathcal{J}ac]_{ij} $ by means of the formula $\mbox{Ber} (X) = \det (A) \det (D- CA^{-1}B)^{-1}$ that reduces the computation of the Berezinian of a super matrix $X$ to a computation of determinants of ordinary matrices. It can be easily checked that for every non-empty intersection $\mathcal{U}_i \cap \mathcal{U}_j$ one gets 
\bear
\det A = - \frac{1}{z_{ij}^{n+1}}, \qquad \qquad \det (D- CA^{-1}B)^{-1} = - {z_{ij}^{n+1}},
\eear
so that $\mbox{Ber}\, [\mathcal{J}ac]_{ij} = 1$, proving triviality of $\mathcal{B}er_{\proj n_\Pi}$. The interested reader might find in the Appendix some explicit computation performed in the intersection $\mathcal{U}_0 \cap \mathcal{U}_1.$\\
Finally, the case $n = 1$ is trivial, as $\proj 1_\Pi$ is split: in general for a \emph{split} supermanifold $\mathcal{B}er_\mani \cong K_{\manir} \otimes \det \mathcal{F}^\ast_\mani$ where $K_{\manir}$ is the canonical sheaf of the reduced manifold, so that in the case of $\proj 1_{\Pi}$ one gets 
\bear
\mathcal{B}er_{\proj 1_\Pi} \cong K_{\proj 1}  \otimes (\Omega^1_{\proj 1})^\ast \cong \mathcal{O}_{\proj 1} (-2) \otimes \mathcal{O}_{\proj 1} (+2) \cong \mathcal{O}_{\proj 1}.
\eear 
thus concluding the proof.
\end{proof}
\noindent Before we move to the next section some remarks are in order. \begin{enumerate}
\item We did already know examples of \emph{split} Calabi-Yau supermanifolds in every bosonic/even dimension: these are the well-known projective superspaces of the kind $\proj {n|n+1}$ for every $n\geq 1$. The previous theorem characterises $\Pi$-projective spaces as relatively simple examples of \emph{non-projected} Calabi-Yau supermanifold for every bosonic/even dimension. In this context, the $\Pi$-projective line $\proj 1_\Pi$ is actually the only Calabi-Yau supermanifold of dimension $1|1$ which has $\proj 1$ as reduced space.
\item The proof of the previous theorem might appear inelegant and somehow cumbersome as in the case $n \geq 2$ it requires explicit knowledge of transition functions of the supermanifold in order to confirm the triviality of the Berezinian sheaf. Still, to the best knowledge of the author, this is actually unavoidable when dealing with non-projected supermanifolds, as the machinery of adjunction theory has not been developed yet and it is not available in this context. In order to answer more sophisticated questions and for classification issues, it would definitely be useful and necessary to fill this gap and provide a suitable adjunction theory for non-projected supermanifolds.       
\end{enumerate}

\section{A Glimpse at $\Pi$-Grassmannians}

In the previous section we have shown that $\Pi$-projective spaces arise as certain non-projected supermanifolds whose fermionic sheaf is related to the cotangent sheaf of their reduced manifold $\proj n$ and, as such, they have a very simple structure. 

Remarkably, something very similar happens more in general to $\Pi$-Grassmannians (see \cite{Manin}), backing the idea of a close connection between $\Pi$-symmetry in supergeometry and the ordinary geometry of cotangent sheaves of the ordinary reduced variety. Indeed we claim  
\begin{center}
\emph{\virgolette all of the $\Pi$-Grassmannians $G_{\Pi}(n,m)$ can be constructed as higher-dimensional non-projected supermanifolds whose fermionic sheaf is given precisely by the cotangent sheaf of their reduced manifold, the ordinary Grassmannian $G(n,m)$''}. 
\end{center}
The difference that makes things trickier compared to the case of the $\Pi$-projective spaces, is that also higher obstruction classes - not only the fundamental one - might appear, leading to non-projected \emph{and} non-split supermanifolds. \\
The construction of $\Pi$-Grassmannian as non-projected supermanifolds related to te cotangent sheaf $\Omega^{1}_{G(n,m)}$ of the underlying Grassmannian $G (n,m)$, the relation between their dimension, structure and the presence of higher obstructions to splitting will be the subject of a forthcoming paper. 

For the time being, in support of the above claim and as an illustrative example, we analyse the structure of the transition functions in certain big-cells of the $\Pi$-Grassmannian $G_{\Pi} (2,4)$: notice that, as in the ordinary context, this is the first $\Pi$-Grassmannian that is \emph{not} a $\Pi$-projective space.\\
We start considering the reduced manifold, the ordinary Grassmannian $G(2,4)$ and look at the change of coordinates between the big-cells
\begin{align}
\mathcal{Z}_{\mathcal{U}_1} = \left ( \begin{array}{cccc}
1 & 0 & x_{11} & x_{21} \\
0 & 1 & y_{11} & y_{21}  
\end{array}
\right ) \qquad \qquad 
\mathcal{Z}_{\mathcal{U}_2} = \left ( \begin{array}{cccc}
1 & x_{12} & 0 & x_{22} \\
0 & y_{12} & 1 & y_{22}  
\end{array}
\right ).
\end{align} 
By row-operations one easily finds that 
\begin{align}
& x_{12} = - \frac{x_{11}}{y_{11}}, \qquad & x_{22} = x_{21} - \frac{x_{11}y_{21}}{y_{11}}, \\
& y_{12} = \frac{1}{y_11}, \qquad & y_{22} = \frac{y_{21}}{y_{11}}.
\end{align}
In the correspondence $\{ dx_{ij} \leftrightarrow \theta_{ij}, dy_{ij} \leftrightarrow \xi_{ij} \}$ of the local frames of the cotangent sheaf with those of its parity-reversed version $\Pi \Omega^1_{G(2,4)}$ we are concerned with, one has the following transition functions 
\begin{align}
& \theta_{12} = - \frac{\theta_{11}}{y_11} + \frac{x_{11}}{y^2_{11}} \xi_{11}, \qquad  \theta_{22} = \theta_{21} - \frac{y_{21}}{y_{11}}\theta_{11} - \frac{x_{11}}{y_{11}} \xi_{21} + \frac{x_{11}y_{21}}{y_{11}^2} \xi_{11}, \\
& \xi_{12} = - \frac{\xi_{11}}{y_{11}^2}, \qquad  \qquad \qquad \; \xi_{22} = \frac{\xi_{21}}{y_11} - \frac{y_{21}}{y_{11}^2} \xi_{11}. 
\end{align}
Now look at the corresponding change of coordinates in $\mathcal{U}_1 \cap \mathcal{U}_2$ for $G_\Pi (2,4)$: the super big-cells then look like 
\begin{align} 
\mathcal{Z}_{\mathcal{U}_{1}} = \left ( \begin{array}{cccc||cccc}
 1 & 0 & x_{11} & x_{21} & 0 & 0 & \theta_{11} & \theta_{21} \\
0 & 1 & y_{11} & y_{21}  & 0 & 0 & \xi_{11} & \xi_{21} \\ 
\hline 
\hline 
0 & 0 & - \theta_{11} & - \theta_{21} & 1 & 0 & x_{11} & x_{21} \\
0 & 0 & - \xi_{11}  & - \xi_{21} & 0 & 1 & y_{11} & y_{21}
\end{array}
\right )
\end{align}
\begin{align}
\mathcal{Z}_{\mathcal{U}_{2}} = \left ( \begin{array}{cccc||cccc}
 1 & x_{12} & 0 & x_{22} & 0 & \theta_{12} & 0 & \theta_{22} \\
0 & y_{12} &1 & y_{22}  & 0 & \xi_{12} & 0 & \xi_{22} \\ 
\hline 
\hline 
0 & - \theta_{12} & 0  & - \theta_{22} & 1 & x_{12} & 0 & x_{22} \\
0 & - \xi_{12} & 0  & - \xi_{22} & 0 &  y_{12} & 1& y_{22}
\end{array}
\right ).
\end{align}
Again, by acting with row-operations on $\mathcal{U}_1$ one finds the following change of coordinates for $G_{\Pi} (2,4)$ in $\mathcal{U}_1 \cap \mathcal{U}_2$
\begin{align}
& \nonumber  x_{12} = - \frac{x_{11}}{y_{11}} - \frac{\theta_{11}\xi_{11}}{y_{11}^2}, \qquad  x_{22} = x_{21} - \frac{x_{11}y_{21}}{y_{11}} + \frac{\theta_{11} \xi_{21}}{y_11} - \frac{x_11}{y_11^2} \xi_{11}\xi_{21} - \frac{y_{21}}{y_{11}^2} \theta_{11}\xi_{11}, \\ 
& \nonumber y_{12} = \frac{1}{y_{11}}, \qquad \qquad \qquad \quad y_{22} = \frac{y_{21}}{y_{11}} + \frac{\xi_{11}\xi_{21}}{y_{11}^2},\\
& \nonumber \theta_{12} = - \frac{\theta_{11}}{y_{11}} + \frac{x_{11}}{y_{11}^2} \xi_{11},  \qquad \theta_{22} = \theta_{21} - \frac{y_{21}}{y_{11}}\theta_{11} - \frac{x_{11}}{y_{11}}\xi_{21} + \frac{x_{11}y_{21}}{y_{11}^2} \xi_{11} - \frac{\theta_{11}\xi_{11}\xi_{21}}{y_{11}^2} \\
& \nonumber \xi_{12} = - \frac{\xi_{11}}{y_{11}^2}, \qquad \qquad \qquad \; \xi_{22} = \frac{\xi_{21}}{y_{11}} - \frac{y_{21}}{y_{11}^2}\xi_{11}.
\end{align}
We observe the following facts: as in the the case of $\Pi$-projective spaces, the bosonic transition functions get nilpotent \virgolette corrections'' taking values in $Sym^2 \Pi \Omega^1_{G (2,4)} (\mathcal{U}_1 \cap \mathcal{U}_2)$. \\
More important, here is the difference: the fermionic transition functions are \emph{almost} the same but \emph{not} actually the same as those of $\Pi \Omega^{1}_{G (2,4)}$ above! Indeed, in the transition functions of $\theta_{22}$ appears a term taking values in $Sym^3 \Pi \Omega^1_{G (2,4)} (\mathcal{U}_1 \cap \mathcal{U}_2)$ - the term $- \frac{\theta_{11}\xi_{11}\xi_{21}}{y_{11}^2}$ - that tells that $G_{\Pi } (2,4)$ will also be characterised by the presence of an higher fermionic obstructions! \\
The study of the geometry of $\Pi$-Grassmannians as non-projected supermanifolds, their characterising algebraic-geometric invariants, and their special relationship with the cotangent sheaf of their underlying manifolds will be discussed in a dedicated follow-up paper.

\appendix 

\section{}

\noindent As the construction is not readily available in literature, we clarify in what follows the structure of the maps entering the second exterior power of a short exact sequence of locally-free sheaf of $\mathcal{O}_{X}$-modules / vector bundles, where $X$ is an ordinary complex manifold. We will work in full generality, even if the for the purpose of the paper it is enough to consider the (easier) special case in which the quotient sheaf is invertible.  

We start looking at the following exact sequence of locally-free sheaf of $\mathcal{O}_{X}$-modules: 
\bear
\xymatrix@R=1.5pt{
0 \ar[r] & \mathcal{F} \ar[r]^\iota & \mathcal{G} \ar[r]^\pi & \mathcal{H} \ar[r] &  0. 
}  
\eear
Then, in general, there is an exact sequence
\bear
\xymatrix@R=1.5pt{
0 \ar[r] & \bigwedge^2 \mathcal{F} \ar[r]^{\wedge^2 \iota} & \bigwedge^2 \mathcal{G} \ar[r]^\phi & \mathcal{Q} \ar[r] &  0. 
}  
\eear
where the map $\phi$ has yet to be defined and the quotient bundle fits into 
\bear
\xymatrix@R=1.5pt{
0 \ar[r] & \mathcal{F} \otimes \mathcal{H} \ar[r] & \mathcal{Q} \ar[r] & \bigwedge^2 \mathcal{H} \ar[r] &  0.
}  
\eear
Indeed, to get an idea, \emph{locally}, the first exact sequence splits to give $\mathcal{G} \cong \mathcal{F} \oplus \mathcal{H}$. Taking the second exterior power one gets 
\bear
\bigwedge^2 \mathcal{G} = \bigwedge^2 \mathcal{F} \oplus \left ( \mathcal{F} \oplus \mathcal{H} \right ) \oplus \bigwedge^2 \mathcal{H},
\eear 
therefore, keep working locally, taking the quotient by $\bigwedge^2 \mathcal{F}$ it gives 
\bear
\mathcal{Q} \cong \slantone{\bigwedge^2 \mathcal{G}}{\bigwedge^2 \mathcal{F}} \cong \left ( \mathcal{F} \otimes \mathcal{H} \right) \oplus \bigwedge^2 \mathcal{H}, 
\eear
that suggests why the second exact sequence is true.

Notice that if we consider the case $\mbox{rank}\, \mathcal{H} = 1$, then one has $\bigwedge^2 \mathcal{H} = 0$, and then sequence for $\mathcal{Q}$ tells that $\mathcal{Q} \cong \mathcal{F } \otimes \mathcal{H},$ therefore one finds that 
\bear
\xymatrix@R=1.5pt{
0 \ar[r] & \bigwedge^2 \mathcal{F} \ar[r]^{\wedge^2 \iota} & \bigwedge^2 \mathcal{G} \ar[r]^\phi & \mathcal{F} \otimes \mathcal{H} \ar[r] &  0. 
}  
\eear
Getting back to the general setting, in order to define the map $\phi$ we consider the alternating map
\bear
\xymatrix@R=1.5pt{
\Phi : \mathcal{G} \otimes \mathcal{G} \ar[rr] & & \mathcal{H} \otimes \mathcal{G} \\
  g_1 \otimes g_2 \ar@{|->}[rr] && \pi (g_1) \otimes g_2 - \pi (g_2) \otimes g_1.
}  
\eear
Notice that one has the commutative diagram
\bear
\xymatrix{
\mathcal{G} \otimes \mathcal{G} \ar[rr]^{\Phi} \ar[d]_q && \mathcal{H} \otimes \mathcal{G} \\
\bigwedge^2 \mathcal{G} \ar[urr]^{\phi_2} 
}  
\eear
and, by the usual universal property, the map $\Phi$ factors over $\bigwedge^2 \mathcal{G}$ to induce a map 
\bear
\xymatrix@R=1.5pt{
\Phi_2 : \mathcal{G} \otimes \mathcal{G} \ar[rr] & & \mathcal{H} \otimes \mathcal{G} \\
  g_1 \wedge g_2 \ar@{|->}[rr] && \pi (g_1) \otimes g_2 - \pi (g_2) \otimes g_1.
}  
\eear
We note that $\Phi_2 \circ \wedge^2 \iota = 0$, indeed: 
\begin{align}
\left (\Phi_2 \circ \wedge^2 \iota \right ) (f_1 \wedge f_2 ) = \Phi_2 (f_1 \wedge f_2 ) = \pi (f_1) \otimes f_2 - \pi (f_2) \otimes f_1 = 0, 
\end{align}
since $\iota : \mathcal{F} \rightarrow \mathcal{G} $ is an inclusion and since $\ker \pi = \mbox{im} \, \iota \cong \mathcal{F}.$ \\
Now, since $\Phi_2 \circ \wedge^2 \iota = 0$, one has that $\Phi_2  (\bigwedge^2 \mathcal{F}) = 0$, so that one has that, in turn $\Phi_2 $ factors through $\mathcal{Q},$ as to yield a well-defined map $\phi_2 : \mathcal{Q} \rightarrow \mathcal{F} \otimes \mathcal{H},$ as follows
\bear
\xymatrix{
\mathcal{G} \otimes \mathcal{G} \ar[rr]^{\Phi} \ar[d]_{q} && \mathcal{H} \otimes \mathcal{G} \\
\bigwedge^2 \mathcal{G} \ar[urr]^{\Phi_2} \ar[d]^{\phi} \\
\mathcal{Q} \ar[uurr]_{\phi_2} 
}  
\eear 
Now, let us examine the following exact sequence obtained by tensoring with $\mathcal{H} \otimes -$ the first exact sequence: 
\bear
\xymatrix@R=1.5pt{
0 \ar[r] & \mathcal{H} \otimes \mathcal{F} \ar[r]^{1 \otimes \iota} &  \mathcal{H} \otimes \mathcal{G} \ar[r]^{1 \otimes \pi } & \mathcal{H} \otimes \mathcal{H} \ar[r] &  0, 
}  
\eear
where the maps are the obvious ones. We observe that 
\begin{itemize}
\item $\mbox{im} \, 1 \otimes \iota \subset \Phi_2 (\bigwedge^2 \mathcal{G}),$ indeed for $ g\in \mathcal{G}$ such that $\pi (g) = h$ we have that 
\bear
(1 \otimes \iota ) (h \otimes f) = h \otimes \iota (f) = \pi (g) \otimes f - \pi (f) \otimes g = \Phi_2 (g \wedge f)
\eear
as $\pi (f) = 0$ and confusing $\iota (f) $ with $f$ (remember that $\iota$ is an immersion). This implies that $\mathcal{H } \otimes \mathcal{F} \cong \mbox{im}\, 1 \otimes \iota \subset \phi_2 (\mathcal{Q})$.
\item $\mbox{im} \, ((1 \otimes \pi) \circ \Phi_2 ) \subset \bigwedge^2 \mathcal{H}$, indeed 
\begin{align}
(1 \otimes \pi) \left ( \Phi_2 (g_1 \wedge g_2 )\right ) &= (1 \otimes \pi ) \left ( \pi (g_1) \otimes g_2 - \pi (g_2) \otimes g_1 \right ) \nonumber \\ 
& =  \pi (g_1) \otimes \pi (g_2) - \pi (g_2) \otimes \pi (g_1) =  \pi (g_1) \wedge \pi (g_2). 
\end{align}
Conversely, working on the fibers, one has that $\bigwedge^2 \mathcal{H} \subset \mbox{im} \left ( (1\otimes \pi ) \circ \Phi_2 \right )$, so that one concludes that $
\mbox{im} \, \left ((1 \otimes \pi) \circ \Phi_2 \right ) =\bigwedge^2 \mathcal{H} $ that in turn implies that 
\bear
\mbox{im} \, \left ((1 \otimes \pi) \circ \phi_2 \right ) =\bigwedge^2 \mathcal{H}.
\eear
This says that $\phi_2$ maps onto $\bigwedge^2 \mathcal{H}$. 
\end{itemize}
As we have shown that the map $\phi_2$ is such that $\mathcal{F} \otimes \mathcal{H} \subset \mbox{im}\, \phi_2$, and that $\phi_2$ is onto $\bigwedge^2 \mathcal{H}$, by counting the dimensions, $\phi_2 $ is actually also injective, therefore $\mathcal{F} \otimes \mathcal{H} \subset \phi_2 (\mathcal{Q}) $ implies that $\mathcal{F} \otimes \mathcal{G} \subset \mathcal{Q}, $ which establishes the second exact sequence for $\mathcal{Q}.$

\section{}

\noindent In this Appendix we support the proof of Theorem \ref{SCY}, by explicitly working out the Berezinian of the super Jacobian in the intersection $\mathcal{U}_0 \cap \mathcal{U}_1$ of the usual covering of $\proj n$. \\
Given the transition functions of $\proj n_{\Pi}$, as in \ref{transi1} and \ref{transi2}, in the intersection $\mathcal{U}_0\cap \mathcal{U}_1$, the super Jacobian matrix reads 
\bear
[\mathcal{J}ac]_{10} = \left ( \begin{array}{ccc|ccc}
 & & & & &  \\
 & A &  &  & B & \\
 & & & & &  \\
 \hline
 & & & & &  \\
 & C &  &  & D & \\
 &  &  &  &  & 
\end{array}
\right )
\eear
where one has
\begin{align}
& A = \left ( \begin{array}{ccccc} 
- \frac{1}{z_{10}^2} & 0 & \cdots & \cdots & 0 \\
- \frac{z_{20}}{z_{10}^2}-2 \frac{\theta_{10}\theta_{20}}{z_{10}^3} & \frac{1}{z_{10}} & 0 &  \cdots & 0\\
\vdots & & \ddots & & \\
\vdots & & & \ddots  & \\
- \frac{z_{n0}}{z_{10}^2}-2 \frac{\theta_{10}\theta_{{n}0}}{z_{10}^3} & 0 & \cdots & 0 & \frac{1}{z_{10}}
\end{array}
\right ) 
\end{align}
\begin{align}
& B = \left ( \begin{array}{ccccc} 
0 &  \cdots & \cdots & \cdots & 0 \\
\frac{\theta_{20}}{z_{10}^2} & - \frac{\theta_{10}}{z_{10}^2} & 0 &  \cdots & 0 \\
\vdots & & \ddots & & \\
\vdots & & & \ddots  & \\
\frac{\theta_{n0}}{z_{10}^2} & 0 & \cdots & 0 & - \frac{\theta_{10}}{z_{10}^2}
\end{array}
\right ) 
\end{align}
\begin{align}
& C = \left ( \begin{array}{ccccc} 
+2 \frac{\theta_{10}}{z_{10}^3} &  0  & \cdots & \cdots & 0 \\
- \frac{\theta_{20}}{z_{10}^2} + 2 \frac{z_{20}}{z_{10}^3}\theta_{10}  & - \frac{\theta_{10}}{z_{10}^2} & 0 &  \cdots & 0 \\
\vdots & & \ddots & & \\
\vdots & & & \ddots  & \\
- \frac{\theta_{{n}0}}{z_{10}^2} + 2 \frac{z_{n0}}{z_{10}^3}\theta_{10} & 0 & \cdots & 0 & - \frac{\theta_{10}}{z_{10}^2}
\end{array}
\right ) 
\end{align}
\begin{align}
& D = \left ( \begin{array}{ccccc} 
-\frac{1}{z_{10}^2} &  0  & \cdots & \cdots & 0 \\
- \frac{z_{20}}{z_{10}^2}  &  \frac{1}{z_{10}} & 0 &  \cdots & 0 \\
\vdots & & \ddots & & \\
\vdots & & & \ddots  & \\
- \frac{z_{n0}}{z_{10}^2} & 0 & \cdots & 0 & \frac{1}{z_{10}}
\end{array}
\right ). 
\end{align}
Now, clearly $\det (A) = - \frac{1}{z_{10}^{n+1}}$ and one can compute that
\bear
D-CA^{-1}B = \left ( \begin{array}{ccccc}
-\frac{1}{z_{10}^2} & 0 & \cdots & \cdots & 0 \\
- \frac{z_{20}}{z_{10}^2} - \frac{\theta_{10} \theta_{20}}{z_{10}^3} & \frac{1}{z_{10}} & 0 & \cdots & 0 \\
\vdots & & \ddots & & \\
\vdots & & & \ddots  & \\
- \frac{z_{n0}}{z_{10}^2} - \frac{\theta_{10} \theta_{n0}}{z_{10}^3} & 0 & \cdots & 0 & \frac{1}{z_{10}} 
\end{array}
\right )
\eear 
so that one finds $\det (D-CA^{-1}B)^{-1} = -z_{10}^{n+1}.$ Putting the two results together, one has
\bear
\qquad \qquad \mbox{Ber} \, [\mathcal{J}ac]_{10} = \det (A) \det(D-CA^{-1}B)^{-1} = \left (- \frac{1}{z_{10}^{n+1}}\right) \cdot \left (- z_{10}^{n+1} \right ) = 1. 
\eear  



\begin{thebibliography}{9}

\bibitem[CNR]{CNR} S.L. Cacciatori, S. Noja, R. Re, \emph{Non Projected Calabi-Yau Supermanifolds over $\proj 2$}, arXiv:1706.01354

\bibitem[DonWit]{DonWit} R. Donagi, E. Witten, \emph{Supermoduli Space is not Projected}, Proc.Symp.Pure Math. {\bf 90} pp 19-72 (2015)

\bibitem[Green]{Green} P. Green, \emph{On Holomorphic Graded Manifolds}, Proc. Amer. Math. Soc., {\bf 85}, 4, pp 587-590 (1982)

\bibitem[Kw]{Kwok} S. Kwok, \emph{The Geometry of $\Pi$-Projective Sheaves}, J. Geom. Phys., {\bf 86}, pp 134-148 (2014)

\bibitem[Lev1]{Levin1} A.M. Levin, \emph{Supersymmetric Elliptic Curves}, Funct. Anal. Appl., {\bf 21}, 3, pp 243–244 (1987)

\bibitem[Lev2]{Levin2} A.M. Levin, \emph{Supersymmetric elliptic and modular functions}, Funct. Anal. Appl., {\bf 22}, 1, pp 60–61 (1988)

\bibitem[M1]{Manin} Yu.I. Manin, \emph{Gauge Fields and Complex Geometry}, {%
Springer-Verlag}, (1988)

\bibitem[M2]{ManinNC} Yu. I. Manin, \emph{Topics in Noncommutative Geometry}, Princeton University Press, (1991)

\bibitem[OkScSp]{OkScSp} C. Okonek, M. Schneider, H. Spindler, \emph{Vector Bundles on Complex Projective Spaces}, Birkh\"{a}user, (1980)

\bibitem[PenSko]{PenkovSk} I.B. Penkov, I. Skornyakov, \emph{Projectivity and D-affinity of flag supermanifolds}, Uspekhi Math. Nauk, {\bf 40}, 1, pp 211-212 (1985)


\end{thebibliography}
\end{document}